\documentclass[11pt,twoside]{article}
\usepackage{fancyhdr,color,graphicx,amsmath,amsfonts,amssymb,latexsym,bm,indentfirst,cite,amsthm,enumerate}
\usepackage[colorlinks=true,backref=page]{hyperref}
\usepackage[all]{xy}
\usepackage[a4paper,left=25mm,right=25mm,top=35mm,body={155mm,230mm}]{geometry}
\usepackage{pdfpages}
\usepackage{titletoc}
\usepackage{fancyhdr}
\newtheorem*{claim}{\hspace{2em} Claim}

\newtheorem{Theorem}{Theorem}[section]%按章排序号
\newtheorem{Definition}[Theorem]{ Definition}
\newtheorem{Lemma}[Theorem]{ Lemma}

\newtheorem*{theoremm}{Main Theorem}
\newtheorem*{Remark}{Remark}

\makeatletter
\newcommand{\subsectionruninhead}{\@startsection{subsection}{2}{0mm}{-\baselineskip}{-0mm}{\bf\large}}
\newcommand{\subsubsectionruninhead}{\@startsection{subsubsection}{3}{0mm}{-\baselineskip}{-0mm}{\bf\normalsize}}
\makeatother

\makeatletter
\newsavebox{\@brx}
\newcommand{\llangle}[1][]{\savebox{\@brx}{\(\m@th{#1\langle}\)}
	\mathopen{\copy\@brx\kern-0.5\wd\@brx\usebox{\@brx}}}
\newcommand{\rrangle}[1][]{\savebox{\@brx}{\(\m@th{#1\rangle}\)}
	\mathclose{\copy\@brx\kern-0.5\wd\@brx\usebox{\@brx}}}
\makeatother

\begin{document}
\newpage

\title{On the Density of Periodic Measures for Star Vector Fields}

\author{Qimai Sun, Guangwa Wang, Wanlou Wu \footnote{Wanlou Wu is the corresponding author. Wanlou Wu was supported by NSFC 12001245.}}
	
\date{}

\maketitle

\begin{abstract}
   In this paper, we prove that every ergodic hyperbolic invariant measure of a $C^1$ star vector field can be approximated by periodic measures in weak$^*$ topology. This extends a classical result of Katok \cite{Ka} for $C^{1+\alpha}(\alpha>0)$ diffeomorphisms to $C^1$ star vector field of any dimension. 
\end{abstract}

\section{Introduction}
   Dynamics aims to describe the long-term evolution of systems governed by known differential rules. In deterministic nonlinear dynamical systems, a core puzzle is how simple deterministic equations can generate complex trajectories that appear random and exhibit extreme sensitivity to initial conditions--a phenomenon known as chaos. Understanding chaos is not about declaring prediction impossible, but about seeking new forms of order within the apparent disorder. Shifting the perspective from individual trajectories to overall statistical behavior is an effective method for revealing such hidden order. In this endeavor, periodic measures serve as a crucial bridge, connecting the classical concept of periodic orbits with statistical descriptions, and providing a key tool for understanding the complex asymptotic behavior of general dynamical systems--as famously illustrated by the \textquotedblleft period three implies chaos\textquotedblright result \cite{LY75}.
 
   In the ergodic theory of dynamical systems, one way to study a measure is to approximate it by well-understood measures. This is particularly the case in smooth ergodic theory, which is the ergodic theory of differentiable dynamical systems. A prime example of such well-understood measures is the class of periodic measures, which combine geometric information of periodic orbits with statistical and ergodic measure theory. In 1970, Bowen \cite{B70} constructed invariant measures on basic sets of an Axiom A diffeomorphism. For an Axiom A diffeomorphism, Sigmund \cite{Sig70} proved that periodic measures are dense in the set of all invariant measures. Katok \cite{Ka} showed that hyperbolic periodic points are dense in the closure of the basin of a given hyperbolic measure for $C^{1+\alpha}(\alpha>0)$ diffeomorphisms. Ma\~{n}\'{e} \cite{MR} used his Ergodic Closing Lemma to prove that every ergodic measure of a generic diffeomorphism can be approximated in the weak topology by measures supported on periodic orbits. Hirayama \cite{HM03} extended Sigmund's result to nonuniformly hyperbolic dynamical systems. The proof relies on the shadowing lemma for such systems, provided in \cite{Ka}, and on a nonuniform version of the specification property. Abdenur, Bonatti and Crovivier \cite{ABC11} presented a variant of Ma\~{n}\'{e}'s classical result: for $C^1$-generic diffeomorphisms, every invariant measure is the weak limit of convex sums of Dirac measures along periodic orbits. For further approximation properties between ergodic measures and periodic measures of diffeomorphisms, see \cite{LLS09,WW10,LLS14}. For semiflows on Hilbert spaces, Lian and Young \cite{Lian} showed that hyperbolic periodic points are dense in the closure of the basin of a given hyperbolic measure. Ma \cite{M22} extended Lian and Young's result to smooth semiflows on separable Banach spaces. Recently, Li, Liang and Liu \cite{LLL24} proved that for $C^{1+\alpha}$($\alpha>0$) nonuniformly hyperbolic vector fields, every ergodic hyperbolic invariant measure not supported on a singularity can be approximated by periodic measures in weak$^*$ topology. Lu and Wu \cite{LW2025} extended this result to $C^1$ star vector fields on three-dimensional manifolds.
   
   In this paper, we discuss this problem for $C^1$ star vector fields on high-dimensional manifolds. More precisely, we investigate ergodic hyperbolic invariant measures for such systems. We prove that for these vector fields,  every ergodic hyperbolic invariant measure can be approximated by periodic measures in weak$^*$ topology. The main result of the present paper is the following.        
\begin{theoremm}\label{MThm}
   Let $\varphi_t$ be the $C^1$ flow generated by a star vector field $X\in\mathfrak{X}^1(M)$ on a $d$-dimensional compact Riemannian manifold $M$, and let $\mu$ be an ergodic hyperbolic invariant measure that is not supported on a singularity. Then there exists a sequence of periodic measures converging to $\mu$ in the weak$^*$ topology.
\end{theoremm}
   Compared to the diffeomorphism case, singularities in vector fields introduce significant difficulties. Flows with singularities exhibit rich and complex dynamics, as exemplified by the \emph{Lorenz attractor} \cite{Lor63,Guc76}. At singularities, one cannot define the \textbf{linear Poincar\'e flow} (see Definition \ref{Def:linearPoincare}). Consequently, some compactness properties are lost, which obstructs the direct application of certain techniques developed for diffeomorphisms, such as Crovisier's central model \cite{CS10} and the distortion arguments of Pujals-Sambarino \cite{PS00}, to singular vector fields. Even in the absence of singularities, the standard Pesin theory used by Lian and Young \cite{Lian} is not directly applicable here because our vector field is only $C^1$. Our paper is organized as follows. Section \ref{P} covers the necessary background on vector ﬁelds. Section \ref{LP} details the Lyapunov metric and the shadowing lemma. Section \ref{MA} is devoted to proving \textbf{Main Theorem}.

\section{Preliminaries}\label{P}
\subsection{Basic Contents of Vector Fields}\label{BCVF} 	
   Let $M$ be a compact $d$-dimensional $(d\geq3)$ $C^\infty$ Riemannian manifold without boundary. Denote by $\mathfrak{X}^r(M)(r\geq 1)$ the space of all $C^r$ vector fields on $M$. Given $X\in\mathfrak{X}^1(M)$, a point $x\in M$ is called a \emph{singularity} of $X$ if $X(x)=0$. Denote by $\text{Sing}(X)$ the set of all singularities. A point $x\in M$ is \emph{regular} if $X(x)\neq 0$. Let $\varphi^X_t=\{\varphi_t\}_{t\in\mathbb{R}}$ be the $C^1$ flow generated by a vector field $X$. We will simply denote it by $\varphi_t$ when no confusion arises. A regular point $p$ is \emph{periodic} if $\varphi_{t_0}(p)=p$ for some $t_0>0$. A critical point is either a singularity or a periodic point. Denote the \emph{normal bundle} of $X$ by $$\mathcal{N}\triangleq\bigcup_{x\in M\setminus\text{Sing}(X)}\mathcal{N}_x,$$where $\mathcal{N}_{x}$ is the orthogonal complement space of the flow direction $X(x)$, i.e., $$\mathcal{N}_x=\left\{v\in T_{x}M: v\perp X(x)\right\}.$$ For the flow $\varphi_t$ generated by $X$, its derivative with respect to the space variable is called the \textbf{tangent flow} and is denoted by $\Phi_t={\rm d}\varphi_t$. The linear Poincar\'{e} flow is a mathematical tool used in the study of dynamical systems, particularly for analyzing the behavior of flows near trajectories or invariant sets. It is derived from the differential (or tangent map) of the flow and describes the evolution of vectors in the tangent bundle after projecting out the component along the flow direction.    
\begin{Definition}\label{Def:linearPoincare}
   For $x\in M\setminus\text{Sing}(X),v\in\mathcal{N}_x$ and $t\in\mathbb{R}$, the \textbf{linear Poincar\'e flow} $$\psi_t:\mathcal{N}\to\mathcal{N}$$ is defined by $$\psi_t(v)\triangleq\Phi_t(v)-\frac{\langle\Phi_t(v),X(\varphi_t(x)) \rangle}{\left\lVert X(\varphi_t(x))\right\rVert^2} X(\varphi_t(x)).$$ Namely, $\psi_t(v)$ is the orthogonal projection of $\Phi_t(v)$ on $\mathcal{N}_{\varphi_t(x)}$ along the flow direction $X(\varphi_t(x))$.  
\end{Definition}  
     
   Fix $T>0$, the norm $$\lVert\psi_T\rVert=\sup\left\{\lvert\psi_T(v)\rvert:~v\in\mathcal{N},~\lvert v\rvert=1\right\}$$ is uniformly upper bounded on $\mathcal{N}$, although $M\setminus{\rm Sing}(X)$ may be not compact. Denote by $$m(\psi_T)=\inf\left\{\lvert\psi_T(v)\rvert:~v\in\mathcal{N},~\lvert v\rvert=1\right\}$$ the mininorm of $\psi_T$. Since $$m(\psi_T)=\lVert\psi_T^{-1}\rVert^{-1}=\lVert\psi_{-T}\rVert^{-1},$$ the mininorm $m(\psi_T)$ is uniformly bounded away from $0$ on $\mathcal{N}$. A closely related construction is the \textbf{scaled linear Poincar\'e flow} $\psi_t^*:\mathcal{N}\to\mathcal{N}$, which is the linear Poincar\'{e} flow scaled by the flow speed.     
\begin{Definition}
   The \textbf{scaled linear Poincar\'e flow} $\psi_t^*:\mathcal{N}\to\mathcal{N}$ is defined by $$\psi^*_t(v)\triangleq\frac{\left\lVert X(x)\right\rVert}{\left\lVert X(\varphi_t(x))\right\rVert}\psi_t(v)=\frac{\psi_t(v)}{\left\lVert\Phi_t|_{\langle X(x)\rangle}\right\rVert},$$ where $x\in M\setminus\text{Sing}(X),~v\in\mathcal{N}_x$ and $\langle X(x)\rangle$ denotes the $1$-dimensional subspace of $T_xM$ spanned by the flow direction $X(x)$.	
\end{Definition}   
   
   The linear Poincar\'e flow $\psi_t$ fails to be compact due to the existence of singularities. To overcome this difficulty, the linear Poincar\'e flow can also be defined in a more general way by Liao \cite{Liao89}. Li, Gan and Wen \cite{LGW05} used the terminology of \textquotedblleft extended linear Poincar\'e flow\textquotedblright. For every point $x\in M$, the \textbf{unit sphere fiber} at $x$ is defined as $$S_xM=\left\{v:~v\in T_xM,~\left\lvert v\right\rvert=1\right\}.$$ The \textbf{sphere bundle} $$SM=\bigcup_{x\in M}S_xM$$ is then compact. For every $v\in SM$, one defines the \textbf{unit tangent flow} $\Phi_t^I: SM\to SM$ as $$\Phi_t^I(v)=\dfrac{\Phi_t(v)}{\left\lvert\Phi_t(v)\right\rvert}.$$ Given a compact invariant set $\Lambda$ of the flow $\varphi_t$, let $$\widetilde{\Lambda}=\text{Closure}\left(\bigcup_{x\in \Lambda\backslash\text{Sing}(X)}\dfrac{X(x)}{\left\lvert X(x)\right\rvert}\right)$$ in $SM$. Thus the essential difference between $\widetilde{\Lambda}$ and $\Lambda$ is on the singularities. We can get more information on $\widetilde{\Lambda}$: it tells us how regular points in $\Lambda$ accumulate singularities. For each $x\in M$, and any two orthogonal vectors $v_1\in S_xM$, $v_2\in T_xM$, define $$\Theta_t(v_1,v_2)=\left(\Phi_t(v_1),\Phi_t(v_2)-\frac{\langle\Phi_t(v_1),\Phi_t(v_2)\rangle}{\left\lVert\Phi_t(v_1)\right\rVert^2} \Phi_t(v_1)\right).$$ By the definition, the two components of $\Theta_t$ remain orthogonal. Denoting $$\Theta_t=\left(\text{Proj}_1(\Theta_t),\text{Proj}_2(\Theta_t)\right),$$ we have, for each regular point $x\in M$ and every vector $v\in \mathcal{N}_x$, $$\psi_t(v)=\text{Proj}_2(\Theta_t)(X(x),v).$$ By the continuity of $\Theta_t$, one can extend the definition of $\psi_t$ to  include singularities. For every vector $u\in\widetilde{\Lambda}$, define the fiber $$\widetilde{\mathcal{N}}_u=\{v\in T_{\pi(u)}M:~v\bot u\},$$ which yields a $(d-1)$-dimensional vector bundle $\widetilde{\mathcal{N}}$ over the base space $\widetilde{\Lambda}$. For every $u\in\widetilde{\Lambda}$, and every $v\in\widetilde{\mathcal{N}}_u$, define the linear map $$\widetilde{\psi}_t(v)=\text{Proj}_2(\Theta_t)(u,v).$$ Then, the original linear Poincar\'{e} flow $\psi_t$ can be \textquotedblleft embedded\textquotedblright in the flow $\Theta_t$. By definition, $\text{Proj}_2(\Theta_t)$ induces a continuous linear flow $\widetilde{\psi}_t$ on the bundle $\widetilde{\mathcal{N}}$. Consequently, $\widetilde{\psi}_t(v)$ varies continuously with respect to the vector field $X$, the time $t$ and the vector $v$ and can be viewed as a compactification of $\psi_t$.
   
   Now, we recall the definitions of several types of splittings for the normal bundle of vector fields.
\begin{Definition}\label{DDS}
   Let $\Lambda$ be a compact invariant set of vector field $X$, and let $\mathcal{N}_\Lambda=E\oplus F$ be an invariant splitting with respect to the linear Poincar\'{e} flow $\psi_t$ over $\Lambda$. The splitting $\mathcal{N}_\Lambda=E\oplus F$ is called a \textbf{dominated splitting} with respect to $\psi_t$, if there exist constants $C\geq 1$ and $\lambda>0$ such that for every $x\in\Lambda$ and $t\geq 0$, $$\left\lVert\psi_t|_{E(x)}\right\rVert\cdot \left\lVert\psi_{-t}|_{F(\varphi_t(x))}\right\rVert\leq Ce^{-\lambda t}.$$ For simplicity, we denote this by $E\prec F$.
\end{Definition} 

\begin{Definition}\label{FHS}
   Let $\Lambda$ be an invariant set of vector field $X$, and let $\mathcal{N}_\Lambda=E\oplus F$ be an invariant splitting with respect to the linear Poincar\'{e} flow $\psi_t$ over $\Lambda$. The splitting $\mathcal{N}_\Lambda=E\oplus F$ is called a \textbf{hyperbolic splitting} with respect to $\psi_t$, if there exist constants $C\geq 1$ and $\lambda>0$ such that $$\left\lVert\psi_t|_{E(x)}\right\rVert\leq Ce^{-\lambda t}\text{ and }\quad\left\lVert\psi_{-t}|_{F(x)}\right\rVert\leq Ce^{-\lambda t},\text{ for every $x\in\Lambda$ and every $t\geq 0$}.$$ In this case, the set $\Lambda$ itself is said to be hyperbolic with respect to $\psi_t$. 
\end{Definition}

\begin{Remark}
   Recall that the scaled linear Poincar\'{e} flow $\psi^*_t$ is induced by the linear Poincar\'{e} flow $\psi_t$. Correspondingly, one can define dominated and hyperbolic splittings with respect to $\psi^*_t$. Specifically, let $\Lambda$ be an {\rm(}not necessarily compact{\rm)} invariant set of flow $\varphi_t$. An invariant splitting $$\mathcal{N}_\Lambda=E\oplus F$$ over the invariant set $\Lambda$ is called a \textbf{dominated splitting} with respect to the scaled linear Poincar\'{e} flow $\psi^*_t$, if there exist constants $C\geq 1$ and $\lambda>0$ such that for every $x\in\Lambda$ and every $t\geq 0$, $$\left\lVert\psi^*_t|_{E(x)}\right\rVert\cdot\left\lVert\psi^*_{-t}|_{F(\varphi_t(x))}\right\rVert\leq Ce^{-\lambda t}.$$ An invariant splitting $$\mathcal{N}_\Lambda=E\oplus F$$ over the invariant set $\Lambda$ is called a \textbf{hyperbolic splitting} with respect to the scaled linear Poincar\'{e} flow $\psi^*_t$, if there exist constants $C\geq 1$ and $\lambda>0$ such that for every $x\in\Lambda$ and every $t\geq 0$, $$\left\lVert\psi^*_t|_{E(x)}\right\rVert\leq Ce^{-\lambda t}\text{ and }\left\lVert\psi^*_{-t}|_{F(x)}\right\rVert\leq Ce^{-\lambda t}.$$ Furthermore, a splitting $$\mathcal{N}_\Lambda=E\oplus F$$ is a \textbf{dominated {\rm (}resp. hyperbolic{\rm)} splitting} with respect to $\psi_t$ \textbf{if and only if} it is a \textbf{dominated {\rm (}resp. hyperbolic{\rm)} splitting} with respect to $\psi^*_t$.     
\end{Remark}

   The notion of star system originated in the study of the famous stability conjecture. Recall that classical theorems of Smale \cite{SS70} (for diffeomorphisms) and Pugh-Shub \cite{PS70} (for flows) state that Axiom A plus the no-cycle condition implies the $\Omega$-stability. Palis and Smale \cite{PaS70} conjectured that the converse also holds, which has become known as the $\Omega$-stability conjecture. During the investigation of this conjecture, Pliss, Liao, and Ma\~{n}\'{e} identified an important condition, termed (by Liao) the star condition. Subsequently, many mathematicians turned their attention to the star system.    
\begin{Definition}\label{SVF}
   A vector field $X\in\mathfrak{X}^1(M)$ is called a \textbf{star vector field}, if there exists a $C^1$ neighborhood $\mathcal{U}$ of $X$ such that for every vector field $Y\in\mathcal{U}$, all singularities and all periodic orbits of $\varphi^Y_t$ are hyperbolic.
\end{Definition} 
   The set of all star vector fields on $M$ is denoted by $\mathfrak{X}^*(M)$, endowed with $C^1$-topology. Indeed, the $\Omega$-stability implies the star condition easily (Franks \cite{FJ71} and Liao \cite{LST79}). Thus, whether the star condition can give back Axiom A plus the no-cycle condition became a striking problem, raised independently by Liao \cite{LST81} and Ma\~{n}\'{e} \cite{MR}. For diffeomorphisms, it is known that a diffeomorphism is star if and only if it is hyperbolic, i.e., satisfies Axiom A plus the no-cycle condition \cite{AN92,HS92}. Gan and Wen \cite{GW06} proved that nonsingular star vector fields satisfy Axiom A plus the no-cycle condition. However, a singular star vector field may fail to satisfy Axiom A, as illustrated by the famous Lorenz attractor \cite{Guc76}. To describe the geometric structure of the Lorenz attractor, Morales, Pacifico and Pujals \cite{MPP04} developed a notion called singular hyperbolicity (see \cite{LGW05} and \cite{MM08} for higher dimensions). Shi, Gan and Wen \cite{SGW14} later showed that a generic star vector field is singular hyperbolic under certain assumptions on its singularities. For a long time, it was believed that singular hyperbolicity was the appropriate notion to characterize generic star vector fields, analogous to hyperbolicity to star diffeomorphisms. Recently, however, da Luz \cite{SFA20} constructed a remarkable example of a $5$-dimensional star vector field that has two singularities with different indices robustly contained in a single chain class; consequently, it is robustly nonsingular hyperbolic. In response to this example, Bonatti and da Luz \cite{BA21} developed a new notion called multisingular hyperbolicity. They proved that every $X$ in an open dense subset of $\mathfrak{X}^*(M)$ is multisingular hyperbolic. Conversely, every multisingular hyperbolic vector field is a star vector field.
   
\subsection{The Lyapunov Exponents of Vector Fields}\label{LEF} 
   In this section, we introduce some ergodic theory for vector fields. Given a vector field $X\in\mathfrak{X}^r(r\geq1)$, a measure $\mu$ is called \textbf{invariant} with respect to the flow $\varphi_t$ (or the vector field $X$), if $\mu$ is an invariant measure of $\varphi_T$, for every $T\in\mathbb{R}$. Similarly, a measure $\mu$ is called \textbf{ergodic} with respect to the flow $\varphi_t$ (or the vector field $X$), if $\mu$ is an ergodic invariant measure of $\varphi_T$, for every $T\in\mathbb{R}$. The sets of all invariant measures and all ergodic measures of the flow $\varphi_t$ are denoted by $\mathcal{M}(X)$, $\mathcal{E}(X)$ respectively. Let $\mu$ be an invariant measure of flow $\varphi_t$. By the Oseledec Multiplicative Ergodic Theorem \cite{OV}, for $\mu$-almost every $x\in M$, there exist a positive integer $k(x)\in[1,d]$, real numbers $\chi_1(x)<\chi_2(x)<\cdots<\chi_{k(x)}(x)$, and a measurable $\Phi_t$-invariant splitting $$T_xM=E_1(x)\oplus E_2(x)\oplus\cdots\oplus E_k(x)$$ such that $$\lim_{t\to\pm\infty}\dfrac{1}{t}\log
   \left\lVert\Phi_t(v)\right\rVert=\chi_i(x),\text{ for all }v\in E_i(x)\setminus\{\mathbf{0}\},\text{ and }i=1,2,\cdots,k(x).$$ The numbers $\chi_1(x),\cdots,\chi_{k(x)}(x)$ are called the \textbf{Lyapunov exponents} at point $x$ of $\Phi_t$ with respect to $\mu$ and the vector formed by these numbers (counted with multiplicity, endowed with the increasing order) is called the \textbf{Lyapunov vector} at point $x$ of $\Phi_t$ with respect to $\mu$. The \textbf{index} of $\mu$, denoted by $\text{Ind}(\mu)$, is defined as $$\text{Ind}(\mu)\triangleq\sum_{\chi_i(x)<0}\text{dim}E_i(x).$$ If $\mu$ is ergodic, then the numbers $k(x),\chi_1(x),\cdots,\chi_{k(x)}(x)$ are constant $\mu$-almost everywhere; we denote them simply by $k,\chi_1,\cdots,\chi_k$. By the Poincar\'{e} Recurrence Theorem, for $\mu$-almost every $x\in M$, $$\lim_{t\to\pm\infty}\dfrac{1}{t}\log\left\lVert\Phi_t|_{\langle X(x)\rangle}\right\rVert=0,$$ where $\langle X(x)\rangle$ is the $1$-dimensional subspace of $T_xM$ spanned by the flow direction $X(x)$. Hence, a zero Lyapunov exponent is always present in the Lyapunov spectrum of the flow $\Phi_t$, associated with the flow direction itself.
  
   For an ergodic invariant measure $\mu$ which is not supported on $\text{Sing}(X)$, the Oseledec Multiplicative Ergodic Theorem applied to the linear Poincar\'e flow $\psi_t:\mathcal{N}\to\mathcal{N}$ implies that for $\mu$-almost every $x\in M\backslash{\rm Sing}(X)$, there exist an integer $k\in[1,d-1]$, real numbers $\chi_1<\chi_2<\cdots<\chi_k$, and a measurable $\psi_t$-invariant splitting $$\mathcal{N}=E_1\oplus E_2\oplus\cdots\oplus E_k$$ of the normal bundle such that $$\lim_{t\to\pm\infty}\dfrac{1}{t}\log\left\lVert\psi_t(v)\right\rVert=\chi_i,~\text{ for all }v\in E_i\setminus \{\mathbf{0}\},\text{ and }i=1,2,\cdots,k.$$ Accordingly, the numbers $\chi_1,\cdots,\chi_k$ are the \textbf {Lyapunov exponents} of $\psi_t$ with respect to $\mu$ and the vector formed by these numbers (counted with multiplicity, endowed with the increasing order) is called the \textbf{Lyapunov vector} of $\psi_t$ with respect to $\mu$. Recall that the scaled linear Poincar\'{e} flow $\psi_t^*:\mathcal{N}\to\mathcal{N}$ is defined by  $$\psi^*_t(v)=\frac{\left\lVert X(x)\right\rVert}{\left\lVert X(\varphi_t(x))\right\rVert}\psi_t(v)=\frac{\psi_t(v)}{\left\lVert\Phi_t|_{\langle X(x)\rangle}\right\rVert},\text{ for any non-zero } v\in\mathcal{N}.$$ Since the Lyapunov exponent of $\Phi_t$ along the flow direction is zero, i.e., $$\lim_{t\to\pm\infty}\dfrac{1}{t}\log\left\lVert\Phi_t|_{\langle X(x)\rangle}\right\rVert=0,$$ it follows that for any nonzero $v \in \mathcal{N}_x$, $$\lim_{t\to\pm\infty}\dfrac{1}{t}\log\left\lVert\psi^*_t(v)\right\rVert=\lim_{t\to\pm\infty}\dfrac{1}{t}\left(\log\left\lVert\psi_t(v)\right\rVert-\log\left\lVert\Phi_t|_{\langle X(x)\rangle}\right\rVert\right)=\lim_{t\to\pm\infty}\dfrac{1}{t}\log\left\lVert\psi_t(v)\right\rVert.$$ Consequently, for $\mu$-almost every $x$, the scaled linear Poincar\'{e} flow $\psi^*_t$ admits exactly the same Oseledets splitting $$\mathcal{N}_x=E_1(x)\oplus E_2(x)\oplus\cdots\oplus E_k(x)$$
   with the same Lyapunov exponents $\chi_1,\cdots,\chi_k$. Hence, the Lyapunov exponents (and the Lyapunov vector) of the scaled linear Poincar\'{e} flow $\psi^*_t$ with respect to $\mu$ coincide with those of the linear Poincar\'{e} flow $\psi_t$. Ignoring multiplicity, we often denote the Lyapunov exponents of the (scaled) linear Poincar\'{e} flow with respect to $\mu$ by $$\lambda_1\leq\lambda_2\leq\cdots\leq\lambda_{d-1}.$$ It means that $$\lambda_j=\chi_i,\quad\text{ for each } d_1+d_2+\cdots+d_{i-1}<j\leq d_1+d_2+\cdots+d_i,$$ where $d_i = \dim E_i(x)$.
    
\begin{Definition}\label{Def:hyperbolicmeasure}
   An ergodic measure $\mu$ of the flow $\varphi_t$ is called \textbf{regular} if it is not supported on a singularity. A regular ergodic measure is called \textbf{hyperbolic} if all Lyapunov exponents of the linear Poincar\'{e} flow $\psi_t$ with respect to $\mu$ are nonzero.
\end{Definition}
     
\begin{Remark}
   One could also define hyperbolicity of an ergodic measure using the tangent flow $\Phi_t={\rm d}\varphi_t$ as usual. However, for every regular ergodic measure, the tangent flow $\Phi_t$ always has a zero Lyapunov exponent in the flow direction.
\end{Remark}

   Given an ergodic hyperbolic invariant measure $\mu$, let $\Gamma$ be the set of $\mu$-generic points that are Lyapunov regular for the linear Poincar\'{e} flow $\psi_t$ (i.e., the full $\mu$-measure set provided by the Oseledec Multiplicative Ergodic Theorem). For every $x\in\Gamma$, let $$\mathcal{N}_x=E_1(x)\oplus E_2(x)\oplus\cdots\oplus E_s(x)\oplus E_{s+1}(x)\oplus\cdots\oplus E_k(x)$$ be the Oseledec decomposition corresponding to the distinct Lyapunov exponents $$\chi_1(\mu)<\chi_2(\mu)<\cdots<\chi_s(\mu)<0<\chi_{s+1}(\mu)<\cdots<\chi_k(\mu),$$ where $k\in[0,d-1]$ is the number of distinct exponents, and let $d_i=\dim E_i(x) \ge 1$ be their multiplicities ($i=1,\cdots,k$). Define the stable and unstable subbundles by $$E^s=E_1\oplus E_2\oplus\cdots\oplus E_s,\qquad E^u=E_{s+1}\oplus E_{s+2}\oplus\cdots\oplus E_k.$$ Then, the Oseledec splitting can be written as $$\mathcal{N}=E^s\oplus E^u,$$ and we call $E^s$ and $E^u$ stable bundle and unstable bundle, respectively.
  
\section{Lyapunov Norms and Shadowing Lemma for Vector Fields}\label{LP}
\subsection{Lyapunov Norms and Lyapunov Charts}
   In this section, we consider only $C^1$ vector fields $X\in\mathfrak{X}^1(M)$. Because the linear Poincaré flow (Definition \ref{Def:linearPoincare}) is not defined at singularities, certain uniform estimates are lost. However, by working with the scaled linear Poincar\'{e} flow, we can recover some uniformity. 
   
   On the tangent space $TM$, the \textbf{Sasaki metric} (see \cite[ Subsection 2.1]{BMW12}), denoted by $d_{S}$, is a natural metric induced by the Riemannian metric of $M$. By the compactness of $M$, there exists a small constant $\rho_S>0$ such that for any two points $x,~y\in M$ with $d(x,y)<\rho_S$, there exists a unique minimizing geodesic joining $y$ to $x$. Moreover, if $u_x\in T_xM$ and $u_y\in T_yM$ satisfy $d_{S}(u_x,u_y)\leq\rho_S$, then letting $u'_y\in T_xM$ be the parallel transport of $u_y$ along this geodesic from $y$ to $x$, one has $$\dfrac{\left(d(x,y)+\lvert u_x-u'_y\rvert\right)}{2}\leq d_{S}(u_x,u_y)\leq 2\left(d(x,y)+\lvert u_x-u'_y\rvert\right).$$ Furthermore, there exists a constant $K_{G^1}$ such that for any two (not necessarily based at the same point) unit vectors $u_1,~u_2\in TM$, $$d_{G^1}(\langle u_1\rangle,\langle u_2\rangle)\leq K_{G^1}\cdot d_{S}(u_1,u_2),$$ where $d_{G^1}$ denotes the Grassmann distance on the bundle $G^1$ of $1$-dimensional subspaces of $TM$, i.e., $$G^1=\left\{N\subset T_xM: N \text{ is a $1$-dimensional linear subspace of $T_xM$}, x\in M\right\}.$$ Since $M$ is compact, there exists a constant $K_0>0$ such that $$\max_{x\in M}\left\{\lvert X(x)\rvert,~\lVert DX(x)\rVert\right\}\leq K_0.$$ Li, Liang, and Liu \cite{LLL24} formulated the precise relationship between the Sasaki metric and the base Riemannian metric as the following Lemma \ref{Dlgx}.         
\begin{Lemma}\label{Dlgx}{\rm\cite[Lemma 3.1]{LLL24}}
   There exist constants $K_1>K_0$ and $\beta^*_0>0$ such that for every $x\in M\backslash\text{Sing}(X)$ and every $y\in B(x,\beta^*_0\lvert X(x)\rvert)$,
\begin{itemize}
   \item[{\rm (1)}] $1-K_1\cdot \dfrac{d(x,y)}{\lvert X(x)\rvert}\leq\dfrac{\lvert X(x)\rvert}{\lvert X(y)\rvert}\leq1+K_1\cdot \dfrac{d(x,y)}{\lvert X(x)\rvert}${\rm;}
	
   \item [{\rm (2)}] $d_S\left(\dfrac{X(x)}{\lvert X(x)\rvert},\dfrac{X(y)}{\lvert X(y)\rvert}\right)\leq K_1\cdot\dfrac{d(x,y)}{\lvert X(x)\rvert}${\rm;}
	
   \item [{\rm (3)}] $d_{G^1}\left(\langle X(x)\rangle,\langle X(y)\rangle\right)\leq K_1\cdot\dfrac{d(x,y)}{\lvert X(x)\rvert}${\rm.} 
\end{itemize}   	
\end{Lemma}

   Let $r_0>0$ be the radius such that the exponential map $$\exp:~TM\to M$$ is a $C^{\infty}$ diffeomorphism when restricted to the disk bundle $T_xM(r_0)=\left\{v\in T_xM: |v|<r_0\right\}$. By reducing $\beta^*_0>0$ if necessary, we may assume that $$10\beta^*_0K_1<\min\{r_0,1\}.$$ Thus, part $(1)$ of Lemma \ref{Dlgx} means that there is no singularity in $B(x,\beta^*_0\lvert X(x)\rvert)$. Next, we analyze the sectional Poincar\'{e} flow in local charts of scaled neighborhoods, the so-called \emph{Liao scaled charts} introduced in \cite{LST79A,LST85,Liao89}, to obtain some uniformity. Similar constructions and discussions can be found in \cite{CY2017,GY,WW19}.
   
   Let $\{e_1,\cdots,e_d\}$ be an orthonormal basis of $\mathbb{R}^d$. For every $x\in M$, choose an orthonormal basis $\{e^x_1,\cdots,e^x_d\}$ of the tangent space $T_xM$. This choice defines a linear isometry $$C_x:\mathbb{R}^d\to T_xM,$$ which provides a metric-preserving coordinate change from the Euclidean space to the tangent space, explicitly given by $$C_x(e_i)=e^x_i,\quad\text{ for $i=1,2,\cdots,d$}.$$ Then, $$\langle C_x(u), C_x(v)\rangle_x=\langle u,v\rangle,\text{ for all }u,v\in\mathbb{R}^d,$$ where $\langle\cdot,\cdot\rangle$ is the standard scalar product on $\mathbb{R}^d$. Using the Riemannian exponential map $\exp_x:T_xM\to M$, we define the composed map $$\text{Exp}_x=\exp_x\circ C_x:\mathbb{R}^d\to M.$$ For every $r>0$, denote by $$\mathbb{R}^d(r)=\left\{v\in\mathbb{R}^d:~\lvert v\rvert\leq r\right\}$$ the closed ball of radius $r$ centered at the origin in $\mathbb{R}^d$. Then, for every $x\in M$, the restriction ${\rm Exp}_x|_{\mathbb{R}^d(r_0)}$ is a $C^{\infty}$ diffeomorphism from $\mathbb{R}^d(r_0)$ onto $B(x,r_0)$. The flow in the neighborhood $B(x,r_0)$ can be locally lifted to $\mathbb{R}^d(r_0)$ via ${\rm Exp}_x$. More precisely, if for a point $y\in\mathbb{R}^d$ and times $t_1<0<t_2$, we have $$\varphi_t({\rm Exp}_x(y))\in B(x,r_0),~\text{ for each $t\in[t_1,t_2]$},$$ then we can define a local flow on $\mathbb{R}^d(r_0)$ by $$\widetilde{\varphi_{x,t}}(y)=\left({\rm Exp}_x|_{\mathbb{R}^d(r_0)}\right)^{-1}\circ\varphi_t\circ{\rm Exp}_x(y),~\text{ for each $t\in[t_1,t_2]$}.$$ In this Liao scaled chart of $x$, the flow $\varphi^X_t$ generated by the vector field $X\in\mathfrak{X}^1(M^d)$ satisfies the differential equation $$\dfrac{dz}{dt}=\widehat{X}_x(z),$$ where $\widehat{X}_x(z)=D\left({\rm Exp}_x|_{\mathbb{R}^d(r_0)}\right)^{-1}\circ X\circ{\rm Exp}_x(z)$. Note that ${\rm Exp}_x(0)=x$ and hence $$\left\lvert\widehat{X}_x(0)\right\rvert=\left\lvert X(x)\right\rvert.$$ Since $M$ is compact, there exists a constant $K_{{\rm E}}>1$ such that $$\max_{x\in M}\left\{\left\lVert D{\rm Exp}_x|_{\mathbb{R}^d(r_0)}\right\rVert,~\left\lVert D\left({\rm Exp}_x|_{\mathbb{R}^d(r_0)}\right)^{-1}\right\rVert\right\}<K_{{\rm E}}.$$ Thus, by enlarging $K_0>0$ if necessary, we may also assume that $$\max_{x\in M}\left\{\max_{p\in\mathbb{R}^d(r_0)}\left\{\left\lvert\widehat{X}_x(p)\right\rvert,~\left\lVert D\widehat{X}_x(p)\right\rVert\right\}\right\}<K_0.$$ Consequently, for every $x\in M$ and every $p\in\mathbb{R}^d(r_0)$, $$\left\lvert\widehat{X}_x(p)-\widehat{X}_x(0)\right\rvert<K_0\lvert p\rvert\leq K_0 r_0.$$
   
   For every regular point $x\in M$, take $$e^x_1=\dfrac{X(x)}{\lvert X(x)\rvert}.$$ Then, in the corresponding coordinates, we have $\widehat{X}_x(0)=\left(\lvert X(x)\rvert,0,0,\cdots,0\right)$. Given $x\in M\backslash\text{Sing}(X)$, for every $y\in B(x,\beta^*_0\lvert X(x)\rvert)$, we may choose an orthonormal basis $\{e^y_1,\cdots,e^y_d\}$ of $T_yM$ with $e^y_1=\dfrac{X(y)}{\lvert X(y)\rvert}$ such that for any two points $y_1,~y_2\in B(x,\beta^*_0\lvert X(x)\rvert)$ and every $j=1,2,\cdots,d$, $$d_S\left(e^{y_1}_j,e^{y_2}_j\right)<2K_1\dfrac{d(y_1,y_2)}{\lvert X(x)\rvert}.$$ Under this orthonormal bases in $T_{B(x,\beta^*_0\lvert X(x)\rvert)}M$, we obtain a uniform estimate on $C^1$ norm of the diffeomorphism ${\rm Exp}_x$, as stated in the following Lemma \ref{Eugx}.   
\begin{Lemma}\label{Eugx}{\rm\cite[Lemma 3.2]{LLL24}}
   There exists a constant $K_2>1$ such that for every $x\in M\backslash\text{Sing}(X)$ and any two points $y_1,y_2\in B(x,\beta^*_0\lvert X(x)\rvert)$, $$\left\lVert\left({\rm Exp}^{-1}_{y_1}\circ{\rm Exp}_{y_2}-id_{\mathbb{R}^d}\right)\big\lvert_{\mathbb{R}^d(r_0/2)}\right\rVert_{C^1}\leq K_2\dfrac{d(y_1,y_2)}{\lvert X(x)\rvert},$$  where $\lVert\cdot\rVert_{C^1}$ denotes the $C^1$ norm.	
\end{Lemma}

\subsection{Shadowing lemma}
   The concept of \textbf{shadowing} was introduced by Sina\u{\i} \cite{Si72}, who proved that Anosov diffeomorphisms satisfy the shadowing property and moreover, every pseudo-orbit is shadowed by a unique true orbit. In essence, shadowing describes the situation in which a true orbit of a dynamical system lies uniformly near a given pseudo-orbit. The shadowing lemma for uniformly hyperbolic diffeomorphisms was fully established by Bowen \cite{B70}, and has since been extensively studied (see, e.g., \cite{CLP89,P99}). For flows, however, the shadowing lemma is considerably more intricate than for diffeomorphisms.
   
   In \cite{LST80}, Liao introduced the notion of \textquotedblleft quasi-hyperbolic strings\textquotedblright, indicating that certain special types of pseudo-orbits can still be shadowed by true orbits even in non-uniformly hyperbolic systems. Liao also provided the first shadowing lemma for single quasi-hyperbolic strings, both for diffeomorphisms \cite{LST79} and for flows \cite{LST80,LST85}, showing that if the head and tail of such a string are sufficiently close, it is shadowed by a periodic orbit. Later, Gan \cite{Gan2002} extended this result to sequences of quasi-hyperbolic strings in the case of diffeomorphisms. Han and Wen \cite{HW18} further generalized Gan’s work to $C^1$ vector fields, establishing a shadowing lemma for sequences of quasi-hyperbolic strings in flows. In what follows, we outline Liao’s shadowing lemma for flows \cite{LST80,LST85}, which forms an essential foundation for the present paper.  
\begin{Definition}
   Let $\Lambda$ be an invariant set of the flow $\varphi_t$ and let $E\subset\mathcal{N}_{\Lambda\backslash\text{Sing}(X)}$ be an invariant subbundle of the scaled linear Poincar\'{e} flow $\psi^*_t$. For constants $C>0$, $\eta>0$ and $T>0$, a point $x\in\Lambda\backslash\text{Sing}(X)$ is called \textbf{$(C,\eta,T,E)$-$\psi^*_t$-contracting} if there exists an increasing sequence of times $0=t_0<t_1<\cdots<t_n<\cdots$ with $t_n\rightarrow+\infty$ and $t_{i+1}-t_i\leq T$, for every $i\in\mathbb{N}$, such that for every $n\in\mathbb{N}$,  $$\prod_{i=0}^{n-1}\left\lVert\psi^*_{t_{i+1}-t_i}|_{E(\varphi_{t_i}(x))}\right\rVert\leq Ce^{-\eta t_n}.$$ A point $x\in\Lambda\backslash\text{Sing}(X)$ is called \textbf{$(C,\eta,T,E)$-$\psi^*_t$-expanding} if it is $(C,\eta,T,E)$-$\psi^*_t$-contracting for the vector field $-X$. 	
\end{Definition}

\begin{Definition}
   Given $\eta>0$, $T>0$. For every $x\in M\backslash\text{Sing}(X)$ and $T_0>T$, the orbit arc $\varphi_{[0,T_0]}$ is called \textbf{$(\eta,T)$-$\psi^*_t$-quasi hyperbolic} with respect to a direct sum splitting $\mathcal{N}_x=E(x)\oplus F(x)$, if there exists a time partition {\rm :} $$0=t_0<t_1<\cdots<t_k=T_0,\text{ with $t_{i+1}-t_i\leq T$ for $i=0,\cdots,k-1$},$$ such that for every $n=0,1,\cdots,k-1$, $$\prod_{i=0}^{n-1}\left\lVert\psi^*_{t_{i+1}-t_i}|_{E(\varphi_{t_i}(x))}\right\rVert\leq e^{-\eta t_n},\quad\prod_{i=n}^{k-1}m\left(\psi^*_{t_{i+1}-t_i}|_{F(\varphi_{t_i}(x))}\right)\geq e^{\eta(t_k-t_n)},$$ and $$\dfrac{\left\lVert\psi^*_{t_{n+1}-t_n}|_{E(\varphi_{t_n}(x))}\right\rVert}{m\left(\psi^*_{t_{n+1}-t_n}|_{F(\varphi_{t_n}(x))}\right)}\leq e^{-\eta(t_{n+1}-t_n)}.$$      	
\end{Definition}

   The following shadowing lemma for singular flows was given by Liao \cite{LST85}, which provides us with a method to find periodic points. 
\begin{Theorem}\label{Shadow}
   Assume that $X\in\mathfrak{X}^1(M)$. Let $\Lambda\subset M\backslash\text{Sing}(X)$ be an invariant set admitting a dominated splitting $\mathcal{N}_\Lambda=E\oplus F$ with respect to the scaled linear Poincar\'{e} flow. Given $\eta>0$ and $T>0$. For $\alpha>0$ and $\varepsilon>0$, there exist a constant $\mathcal{D}=\mathcal{D}(\alpha,\varepsilon)>0$ such that for every $(\eta,T)$-$\psi^*_t$-quasi hyperbolic orbit segment $\varphi_{[0,T_0]}(x)$ satisfying that
\begin{itemize}
   \item $d(x,\text{Sing}(X))>\alpha$ and $d(\varphi_{T_0}(x),\text{Sing}(X))>\alpha${\rm;}
		
   \item $x\in\Lambda$, $\varphi_{T_0}(x)\in\Lambda$ and $d(x,\varphi_{T_0}(x))<\mathcal{D}${\rm;} 
\end{itemize}  
   there exist a strictly increasing $C^1$ function $\theta:[0,T_0]\to\mathbb{R}$ and a periodic point $p\in M\backslash\text{Sing}(X)$ with the following properties:   
\begin{description}
   \item[(1)] $\theta(0)=0$ and $1-\varepsilon<\theta'(t)<1+\varepsilon$, for every $t\in[0,T_0]${\rm;}
		
   \item[(2)] $p$ is a periodic point with period $\theta(T_0)$ {\rm:} $\varphi_{\theta(T_0)}(p)=p${\rm;}
		
   \item[(3)] $d(\varphi_t(x),\varphi_{\theta(t)}(p))<\varepsilon\lvert X(\varphi_t(x))\rvert$, for every $t\in[0,T_0]${\rm.}
\end{description}    
\end{Theorem}

\section{Density of periodic measures: proof of Main Theorem}\label{MA}
   In this section, we prove the \textbf{Main Theorem}. Let $X\in\mathfrak{X}^*(M)$ be a star vector field, and let $\mu\in\mathcal{M}(X)$ be an ergodic hyperbolic invariant measure of $X$. Shi, Gan and Wen \cite{SGW14} proved that $\mu$ is hyperbolic with respect to the scaled linear Poincar\'{e} flow $\psi^*_t$. Consequently, we have the Oseledec hyperbolic splitting $$\mathcal{N}=E^s\oplus E^u.$$ By analyzing the Lyapunov exponents of the scaled linear Poincar\'{e} flow $\psi^*_t$ with respect to $\mu$, we complete the proof of the \textbf{Main Theorem}.   
\begin{proof}[Proof of \textbf{Main Theorem}]
   We divide the proof into two cases based on the Lyapunov exponents of $\psi^*_t$ with respect to $\mu$.
   
   \textbf{Case 1: All exponents have the same sign}. Suppose all Lyapunov exponents of $\psi^*_t$ with respect to $\mu$ are negative (or all are positive). Then $\mu$ is an ergodic measure supported on an attracting (respectively, repelling) periodic orbit. 
    
   \textbf{Case 2: Mixed signs}. Suppose the Lyapunov exponents of $\psi^*_t$ with respect to $\mu$ are neither all negative nor all positive; equivalently, the stable bundle $E^s$ and the unstable bundle $E^u$ are both nontrivial. In this situation, the Oseledec splitting itself exhibits a dominated structure. Precisely, we have the following claim.    
\begin{claim}
   The Oseledec splitting $\mathcal{N}_\Gamma=E^s\oplus E^u$ is a dominated splitting for the scaled linear Poincar\'{e} flow $\psi^*_t$.
\end{claim}
\begin{proof}[Proof of Claim]
   For a nontrivial ergodic measure $\mu$ of a star vector field $X\in\mathfrak{X}^*(M)$, Li, Shi, Wang and Wang \cite{LSWW20} proved that the scaled linear Poincar\'{e} flow $\psi^*_t$ admits a dominated splitting $$\mathcal{N}_{\text{supp}(\mu)\backslash\text{Sing}(X)}=E\oplus F,$$ where $\dim(E)=\text{Ind}(\mu)$. Consequently, it suffices to prove that this dominated splitting coincides with the Oseledets splitting $\mathcal{N}_\Gamma=E^s\oplus E^u$.
   
   Let $$\lambda^-(\mu)=\lim\limits_{t\to\pm\infty}\dfrac{1}{t}\log\left\lVert\psi^*_t|_{E^s}\right\rVert,\qquad\lambda^+(\mu)=\lim\limits_{t\to\pm\infty}\dfrac{1}{t}\log m\left(\psi^*_t|_{E^u}\right).$$ Given a point $x\in\left(\Gamma\cap\text{supp}(\mu)\right)\backslash\text{Sing}(X)$, fix $$\epsilon\in\left(0,\dfrac{\min\{\lvert\lambda^-(\mu)\rvert,~\lambda^+(\mu)\}}{8}\right)$$ sufficiently small. For the Oseledec splitting $\mathcal{N}_x=E^s(x)\oplus E^u(x)$, there exists a positive constant $T_0>0$ such that for every $t\geq T_0$, $$\left\lvert\psi^*_t(v)\right\rvert\leq e^{(\lambda^-(\mu)+\epsilon)t}\lvert v\rvert,\text{ for every nonzero vector }v\in E^s(x),\text{ and }$$ $$m\left(\psi^*_t(v)\right)\geq e^{(\lambda^+(\mu)-\epsilon)t}\lvert v\rvert,\text{ for every nonzero vector }v\in E^u(x).$$ By the definition of dominated splitting, there exist two constants $C\geq 1$ and $\lambda>0$ such that for the splitting $\mathcal{N}_x=E(x)\oplus F(x)$ and every $t\geq 0$, $$\dfrac{\left\lVert\psi^*_t|_{E(x)} \right\rVert}{\left\lVert\psi^*_t|_{F(x)}\right\rVert}\leq Ce^{-\lambda t}.$$ Now suppose, for contradiction, that $$E^s(x)\nsubseteq E(x).$$ Then, there exists a nonzero vector $u\in E(x)\backslash E^s(x)$. On the other hand, since $$\text{dim}(E^s(x))=\text{Ind}(\mu)=\text{dim}(E(x)),$$ we also have $$E(x)\nsubseteq E^s(x).$$ Thus, there exists a nonzero vector $v\in E^s(x)\backslash E(x)$. Therefore, there exists a vector decomposition $$u=u_1+u_2,\text{ where $u_1\in E^s(x)$, $u_2\in E^u(x)$ and $u_2\neq 0$},$$ of $u$ under the splitting $\mathcal{N}_x=E^s(x)\oplus E^u(x)$. Similarly, there exists a vector decomposition $$v=v_1+v_2,\text{ where $v_1\in E(x)$, $v_2\in F(x)$ and $v_2\neq 0$},$$ of $v$ under the splitting $\mathcal{N}_x=E(x)\oplus F(x)$. Consequently, 
\begin{equation*}
\begin{aligned}
   \dfrac{\lvert\psi^*_t(u)\rvert}{\lvert\psi^*_t(v)\rvert}&=\dfrac{\lvert\psi^*_t(u_1)+\psi^*_t(u_2)\rvert}{\lvert\psi^*_t(v)\rvert}\geq\dfrac{\lvert\psi^*_t(u_2)\lvert-\rvert\psi^*_t(u_1)\rvert}{\lvert\psi^*_t(v)\rvert}\geq\dfrac{e^{(\lambda^+(\mu)-\epsilon)t}\lvert u_2\rvert-e^{(\lambda^-(\mu)+\epsilon)t}\lvert u_1\rvert}{e^{(\lambda^-(\mu)+\epsilon)t}\lvert v\rvert}\\&=\dfrac{e^{(\lambda^+(\mu)-\lambda^-(\mu)-2\epsilon)t}\lvert u_2\rvert-\lvert u_1\rvert}{\lvert v\rvert}\rightarrow+\infty\quad\text{(as $t\rightarrow+\infty$)},
\end{aligned}
\end{equation*}
	and 
\begin{equation*}
\begin{aligned}
   \dfrac{\lvert\psi^*_t(v)\rvert}{\lvert\psi^*_t(u)\rvert}&=\dfrac{\lvert\psi^*_t(v_1)+\psi^*_t(v_2)\rvert}{\lvert\psi^*_t(u)\rvert}\geq\dfrac{\lvert\psi^*_t(v_2)\lvert-\rvert\psi^*_t(v_1)\rvert}{\lvert\psi^*_t(u)\rvert}\geq\dfrac{\lvert\psi^*_t(v_2)\lvert-Ce^{-\lambda t}\lvert\psi^*_t(v_2)\lvert}{\lvert\psi^*_t(u)\rvert}\\&=\left(1-Ce^{-\lambda t}\right)\dfrac{\lvert\psi^*_t(v_2)\lvert}{\lvert\psi^*_t(u)\rvert}\geq\left(1-Ce^{-\lambda t}\right)C^{-1}e^{\lambda t}\dfrac{\lvert v_2\rvert}{\lvert u\rvert}\rightarrow+\infty\quad\text{(as $t\rightarrow+\infty$)}.
\end{aligned}
\end{equation*}
   This leads to a contradiction. Therefore, the assumption $E^s(x)\nsubseteq E(x)$ is false, and hence $$E^s(x)=E(x).$$ Because $\dim E^s(x)=\dim E(x)$, it follows that $$E^u(x)=F(x)$$ as well. Thus, the two splittings $\mathcal{N}_\Gamma=E^s\oplus E^u$ and $\mathcal{N}_\Gamma=E\oplus F$ coincide. This completes the proof of the claim.          
\end{proof}	

   Let $C(M)$ be the space of all continuous real-valued functions on $M$, and let $\left\{f_i\right\}_{i=1}^{+\infty}$ be a countable dense subset of $C(M)$. For any two measures $\mu,~\nu\in\mathcal{M}(X)$, define $$d_{\mathcal{M}}(\mu,\nu)=\sum_{i=1}^{+\infty}\dfrac{\left\lvert\int_Mf_id\mu-\int_Mf_id\nu\right\rvert}{2^i\left\lVert f_i\right\rVert}.$$ Then $d_{\mathcal{M}}(\cdot,\cdot)$ is a metric which induces the weak$^*$ topology on $\mathcal{M}(X)$. Given an ergodic hyperbolic invariant regular measure $\mu$, for every $\varepsilon>0$, we will prove that there exists a periodic measure $\mu_p$ such that $$d_{\mathcal{M}}(\mu,\mu_p)<\varepsilon.$$ First, choose a positive integer $n$ sufficiently large so that for every invariant measure $\nu$, the following holds
\begin{equation*}
   \sum_{i=n+1}^{+\infty}\dfrac{\left\lvert\int_Mf_id\mu-\int_Mf_id\nu\right\rvert}{2^i\left\lVert f_i\right\rVert}\leq\sum_{i=n+1}^{+\infty}\dfrac{1}{2^{i-1}}<\dfrac{\varepsilon}{2}.\tag{$\vartriangle$}
\end{equation*}   
   According to the Birkhoff Ergodic Theorem, there exists a $\varphi$-invariant set $\Lambda_B$ of full $\mu$-measure such that for every $x\in\Lambda_B$ and every $f\in C(M)$, $$\lim_{T\rightarrow+\infty}\dfrac{1}{T}\int^{T}_{0}f(\varphi_t(x))dt=\int_Mfd\mu.$$ Thus, there exists a constant $T_1>1$ such that for all $T\geq T_1$, all $x\in\Lambda_B$ and every $i=1,2,\cdots,n$,
\begin{equation*}
   \left\lvert\dfrac{1}{T}\int^{T}_{0}f_i(\varphi_t(x))dt-\int_Mf_id\mu\right\rvert<\dfrac{\varepsilon}{4n}\cdot\min_{1\leq i\leq n}\left\{2^i\left\lVert f_i\right\rVert\right\}. \tag{$\vartriangle\vartriangle$}
\end{equation*}   
	  
   By the claim, the hyperbolic Oseledec splitting $$\mathcal{N}_\Gamma=E^s\oplus E^u$$ of the scaled linear Poincar\'{e} flow with respect to $\mu$ is a dominated splitting. According to \cite[Lemma 5.1]{WYZ}, for every $\epsilon>0$, there exists a constant $T_2=T(\epsilon)>0$ such that for $\mu$-a.e $x\in M$ and every $T\geq T_2$, we have that $$\lim_{J\to+\infty}\dfrac{1}{JT}\sum_{i=0}^{J-1}\log\left\lVert\psi^*_T|_{E^s(\varphi_{iT})}\right\rVert\text{ exists and is contained in }[\lambda^-(\mu),\lambda^-(\mu)+\epsilon),$$ and $$\lim_{J\to+\infty}\dfrac{1}{JT}\sum_{i=0}^{J-1}\log\left\lVert\psi^*_{-T}|_{E^u(\varphi_{-iT})}\right\rVert\text{ exists and is contained in }(-\lambda^+(\mu)-\epsilon,-\lambda^+(\mu)].$$ Let $$\chi=\min\left\{\left\lvert\lambda^-(\mu)\right\rvert,~\lambda^+(\mu)\right\},~T_0=\max\left\{T_1,~T_2\right\},~\eta_0=(\chi-\epsilon/4)T_0.$$ Then there exists a positive integer $N=N(T_0)$ such that for every integer $J\geq N$ and $\mu$-a.e. $x\in M$, we have that $$\prod_{i=0}^{J-1}\left\lVert\psi^*_{T_0}|_{E^s(\varphi_{iT_0}(x))}\right\rVert\leq e^{-(\chi-\epsilon/4)JT_0},~\prod_{i=0}^{J-1}m\left(\psi^*_{T_0}|_{E^u(\varphi_{iT_0}(x))}\right)\geq e^{(\chi-\epsilon/4)JT_0},$$ and $$\dfrac{\left\lVert\psi^*_{T_0}|_{E^s(x)}\right\rVert}{m\left(\psi^*_{T_0}|_{E^u(x)}\right)}\leq e^{-T_0(\chi-\epsilon/4)}.$$ For every $C>0$, the \textbf{Pesin block } $\Lambda^{T_0}_{\eta_0}(C)$ is defined by
\begin{equation*}
\begin{aligned}
   \Lambda^{T_0}_{\eta_0}(C)=\Bigg\{x\in\Gamma:&\prod_{i=0}^{J-1}\left\lVert\psi^*_{T_0}|_{E^s(\varphi_{iT_0}(x))}\right\rVert\leq Ce^{-J\eta_0},~\forall~ J\geq 1,\\ &\prod_{i=0}^{J-1}m\left(\psi^*_{T_0}|_{E^u(\varphi_{iT_0}(x))}\right)\geq C^{-1}e^{J\eta_0},~\forall~ J\geq 1,~d(x,\text{Sing}(X))\geq\dfrac{1}{C}\Bigg\}.
\end{aligned}   
\end{equation*}
   According to \cite[Proposition 5.3]{WYZ}, every set $\Lambda^{T_0}_{\eta_0}(C)$ is compact and $$\mu\left(\Lambda^{T_0}_{\eta_0}(C)\right)\to 1\text{ as }C\to+\infty.$$ 
   
   Denote $$K=\max_{1\leq i\leq n}\left\{\left\lVert f_i\right\rVert\right\}.$$ Take $\gamma>0$ sufficiently small such that $$\left(2K+1\right)\gamma<\dfrac{\varepsilon}{4n}\cdot\min_{1\leq i\leq n}\{2^i\left\lVert f_i\right\rVert\}.$$ Recall that $\left\lvert X(x)\right\rvert\leq K_0$, for every $x\in M$. Since each $f_i$ ($1\leq i\leq n$) is uniformly continuous on $M$, there exists a constant $\xi\in[0,\gamma]$ sufficiently small such that for any two points $x,~y\in M$ with $d(x,y)\leq\xi K_0$ and every $i=1,\cdots,n$, $$\left\lvert f_i(x)-f_i(y)\right\rvert<\gamma.$$ Take $C>0$ sufficiently large such that $\mu\left(\Lambda^{T_0}_{\eta_0}(C)\right)$ can be made arbitrarily close to $1$. For this fixed $C$, there is a positive integer $j_0=j_0(C)\in\mathbb{N}^{+}$ such that $C<e^{\frac{j_0T_0\epsilon}{4}}$. Consequently, for every $x\in\Lambda^{T_0}_{\eta_0}(C)$ and every integer $J\geq 1$, it holds that $$\prod_{i=0}^{J-1}\left\lVert\psi^*_{j_0T_0}|_{E^s(\varphi_{ij_0T_0}(x))}\right\rVert\leq e^{-(\chi-\epsilon/2)Jj_0T_0},~\prod_{i=0}^{J-1}m\left(\psi^*_{j_0T_0}|_{E^u(\varphi_{ij_0T_0}(x))}\right)\geq e^{(\chi-\epsilon/2)Jj_0T_0}.$$ Setting $T=j_0T_0$, $\eta=(\chi-\epsilon/2)j_0T_0$, we consider the set
\begin{equation*}
\begin{aligned}
   \Lambda^{T}_{\eta}(C)=\Bigg\{x\in\Gamma:&\prod_{i=0}^{J-1}\left\lVert\psi^*_{T}|_{E^s(\varphi_{iT}(x))}\right\rVert\leq e^{-J\eta},~\forall~ J\geq 1,\\ &\prod_{i=0}^{J-1}m\left(\psi^*_{T}|_{E^u(\varphi_{iT}(x))}\right)\geq e^{J\eta},~\forall~ J\geq 1,~d(x,\text{Sing}(X))\geq\dfrac{1}{C}\Bigg\}.
\end{aligned}   
\end{equation*}	
   Take a point $x_0\in\Lambda^T_\eta(C)\cap\text{supp}(\mu)$. By the continuity of $X(x)$ and since $X(x_0)\neq0$, there exists a constant $r>0$ sufficiently small such that for every $x\in B(x_0,r)$, $$B(x_0,r)\subset B\left(x,\xi\lvert X(x)\right\rvert).$$ Let $\varepsilon'=\min\left\{\varepsilon,~\xi\right\}$. For every $\alpha\in(0,1/C)$, Theorem \ref{Shadow} provides a positive constant $$\mathcal{D}=\mathcal{D}(\alpha,\varepsilon')>0.$$ Since $\mu\left(B(x_0,r)\cap\Lambda^T_\eta(C)\cap\text{supp}(\mu)\cap\Lambda_B\right)>0$, by the Poincar\'{e} Recurrence Theorem, there exist a point $y\in B(x_0,r)\cap\Lambda^T_\eta(C)\cap\Lambda_B$ and a sufficiently large integer $l$ such that $$\varphi_{lT}(y)\in B(x_0,r)\cap\Lambda^T_\eta(C)\cap\text{supp}(\mu)\cap\Lambda_B\text{ and } d(y,\varphi_{lT}(y))<\mathcal{D}.$$ Thus, we obtain an $(\eta,T)$-$\psi^*_t$-quasi hyperbolic orbit segment $\varphi_{[0,lT]}(y)$ satisfying
\begin{itemize}
   \item $d(y,\text{Sing}(X))>\alpha$ and $d(\varphi_{lT}(y),\text{Sing}(X))>\alpha${\rm;}
		
   \item $y\in\Lambda^T_\eta(C)$, $\varphi_{lT}(y)\in\Lambda^T_\eta(C)$ and $d(y,\varphi_{lT}(y))<\mathcal{D}${\rm.} 
\end{itemize}
   By Theorem \ref{Shadow}, there exist a strictly increasing $C^1$ function $\theta:[0,lT]\to\mathbb{R}$ and a periodic point $p\in M$ such that   
\begin{description}
   \item[(1)] $\theta(0)=0$ and $1-\gamma\leq1-\xi<\theta'(t)<1+\xi\leq1+\gamma$, for every $t\in[0,lT]${\rm;}
		
   \item[(2)] $p$ is a periodic point with period $\theta(lT)$ {\rm:} $\varphi_{\theta(lT)}(p)=p${\rm;}
		
   \item[(3)] $d(\varphi_t(y),\varphi_{\theta(t)}(p))<\varepsilon'\lvert X(\varphi_t(y))\rvert<\xi\lvert X(\varphi_t(y))\rvert$, for every $t\in[0,lT]${\rm.}
\end{description}
   Therefore, for each $i=1,2,\cdots,n$, we have that 
\begin{equation*}
\begin{aligned}
   \left\lvert\dfrac{1}{\theta(lT)}\int_{0}^{\theta(lT)}f_i(\varphi_t(p))dt-\dfrac{1}{lT}\int_{0}^{\theta(lT)}f_i(\varphi_t(p))dt\right\rvert&\leq\left\lvert\dfrac{\theta(lT)}{lT}-1\right\rvert\cdot\dfrac{1}{\theta(lT)}\int_{0}^{\theta(lT)}f_i(\varphi_t(p))dt\\&\leq\gamma K.
\end{aligned}   
\end{equation*} 
   and    
\begin{equation*}
\begin{aligned}
   &\left\lvert\dfrac{1}{lT}\int_{0}^{\theta(lT)}f_i(\varphi_s(p))ds-\dfrac{1}{lT}\int_{0}^{lT}f_i(\varphi_t(x))dt\right\rvert=\dfrac{1}{lT}\left\lvert\int_{0}^{lT}f_i(\varphi_{\theta(t)}(p))d\theta(t)-\int_{0}^{lT}f_i(\varphi_t(x))dt\right\rvert\\\leq&\dfrac{1}{lT}\left\lvert\int_{0}^{lT}f_i(\varphi_{\theta(t)}(p))(\theta'(t)-1)dt\right\rvert+\dfrac{1}{lT}\left\lvert\int_{0}^{lT}\left[f_i(\varphi_{\theta(t)}(p))-f_i(\varphi_t(x))\right]dt\right\rvert\\\leq&\gamma K+\gamma.
\end{aligned}   
\end{equation*} 
   Consequently, $$\left\lvert\dfrac{1}{\theta(lT)}\int_{0}^{\theta(lT)}f_i(\varphi_t(p))dt-\dfrac{1}{lT}\int_{0}^{lT}f_i(\varphi_t(x))dt\right\rvert<\dfrac{\varepsilon}{4n}\cdot\min_{1\leq i\leq n}\{2^i\left\lVert f_i\right\rVert\}.$$
	
   Denote by $\mu_p$ the invariant measure supported on the periodic orbit ${\rm Orb}(p)$. From the equation $(\vartriangle)$ and $(\vartriangle\vartriangle)$, we have that    
\begin{equation*}
\begin{aligned}
   &d_{\mathcal{M}}(\mu,\mu_p)=\sum_{i=1}^{+\infty}\dfrac{\left\lvert\displaystyle\int_Mf_id\mu-\int_Mf_id\mu_p\right\rvert}{2^i\left\lVert f_i\right\rVert}\\&=\sum_{i=1}^{n}\dfrac{\left\lvert\displaystyle\int_Mf_id\mu-\int_Mf_id\mu_p\right\rvert}{2^i\left\lVert f_i\right\rVert}+\sum_{i=n+1}^{+\infty}\dfrac{\left\lvert\displaystyle\int_Mf_id\mu-\int_Mf_id\mu_p\right\rvert}{2^i\left\lVert f_i\right\rVert}\\&\leq\sum_{i=1}^{n}\dfrac{\left\lvert\displaystyle\int_Mf_id\mu-\dfrac{1}{lT}\int_{0}^{lT}f_i(\varphi_t(y))dt+\dfrac{1}{lT}\int_{0}^{lT}f_i(\varphi_t(y))dt-\int_Mf_id\mu_p\right\rvert}{2^i\left\lVert f_i\right\rVert}+\sum_{i=n+1}^{+\infty}\dfrac{1}{2^{i-1}}\\&\leq\sum_{i=1}^{n}\dfrac{\left\lvert\displaystyle\int_Mf_id\mu-\dfrac{1}{lT}\int_{0}^{lT}f_i(\varphi_t(y))dt\right\rvert+\left\lvert\dfrac{1}{lT}\displaystyle\int_{0}^{lT}f_i(\varphi_t(y))dt-\dfrac{1}{\theta(lT)}\int_{0}^{\theta(lT)}f_i(\varphi_t(p))dt\right\rvert}{2^i\left\lVert f_i\right\rVert}+\dfrac{\varepsilon}{2}\\&<\sum_{i=1}^{n}\dfrac{\dfrac{\varepsilon}{4n}\cdot\min\limits_{1\leq i\leq n}\left\{2^i\left\lVert f_i\right\rVert\right\}+\dfrac{\varepsilon}{4n}\cdot\min\limits_{1\leq i\leq n}\left\{2^i\left\lVert f_i\right\rVert\right\}}{2^i\left\lVert f_i\right\rVert}+\dfrac{\varepsilon}{2}\leq\sum_{i=1}^{n}\dfrac{\varepsilon}{2n}+\dfrac{\varepsilon}{2}<\varepsilon.
\end{aligned}   
\end{equation*}
   This completes the proof of \textbf{Main Theorem}.          
\end{proof}

%\noindent\textbf{Acknowledgements.} We would like to thank Dawei Yang for his useful suggestions. We would also like to thank Gang Liao and Lingmin Liao for their useful conversations.
%\clearpage
%\addcontentsline{toc}{section}{Reference}
%\bibliographystyle{amsplain}
%\bibliography{reference}

\begin{thebibliography}{99}
\bibitem{ABC11} F. Abdenur, C. Bonatti, and S. Crovisier, Nonuniform hyperbolicity for $C^1$-generic diffeomorphisms, {\it Israel J. Math.}, {\bf 183} (2011), 1--60.	
	
\bibitem{AN92} N. Aoki, The set of Axiom A diffeomorphisms with no cycles, {\it Bol. Soc. Brasil. Mat. (N.S.)}, {\bf 23} (1992), no. 1-2, 21--65.

\bibitem{BA21} C. Bonatti and A. da Luz, Star flows and multisingular hyperbolicity, {\it J. Eur. Math. Soc.}, {\bf 23} (2021), no. 8, 2649--2705.

\bibitem{B70} R. Bowen, Markov partitions for Axiom A diffeomorphisms, {\it Amer. J. Math.}, {\bf 92} (1970), 725--747.

\bibitem{BMW12} K. Burns, H. Masur, and A. Wilkinson, The weil-petersson geodesic flow is ergodic, {\it Ann. of Math. (2)}, {\bf 175} (2012), no. 2, 835--908.

\bibitem{CLP89} S. N. Chow, X. B. Lin, and K. J. Palmer, A shadowing lemma with applications to semilinear parabolic equations, {\it SIAM J. Math. Anal.}, {\bf 20} (1989), no. 3, 547--557.

\bibitem{CS10} S. Crovisier, Birth of homoclinic intersections: a model for the central dynamics of partially hyperbolic systems, {\it Ann. of Math. (2)}, {\bf 172} (2010), no. 3, 1641--1677.

\bibitem{CY2017} S. Crovisier and D. W. Yang, Homoclinic tangencies and singular hyperbolicity for three-dimensional vector fields, {\it preprint}, (2017), Arxiv:1702.05994.

\bibitem{SFA20} A. da Luz, Star flows with singularities of different indices, {\it preprint}, (2020), ArXiv:1806.09011v2.

\bibitem{FJ71} J. Franks, Necessary conditions for stability of diffeomorphisms,, {\it Trans. Amer. Math. Soc.}, {\bf 158} (1971), 301--308.

\bibitem{Gan2002} S. B. Gan, A generalized shadowing lemma, {\it Discrete Contin. Dyn. Syst.}, {\bf 8} (2002), no. 3, 627--632.

\bibitem{GW06} S. B. Gan and L. Wen, Nonsingular star flows satisfy Axiom A and the no-cycle condition, {\it Invent. Math.}, {\bf 164} (2006), no. 2, 279--315.

\bibitem{GY} S. B. Gan and D. W. Yang, Morse-Smale systems and horseshoes for three dimensional singular flows, {\it Ann. Sci. \'{E}c. Norm. Sup\'{e}r. (4)}, {\bf 51} (2018), no. 1, 39--112.

\bibitem{Guc76} J. Guckenheimer, A strange, strange attractor, in The Hopf Bifurcation and Its Applications, Applied Mathematical Sciences {\bf 19}, Springer New York, New York, NY, 1976, 368--381.

\bibitem{HW18} B. Han and X. Wen, A shadowing lemma for quasi-hyperbolic strings of flows, {\it J. Differential Equations}, {\bf 264} (2018), no. 1, 1--29.

\bibitem{HS92} S. Hayashi, Diffeomorphisms in $\mathcal{F}^1(M)$ satisfy Axiom A, {\it Ergodic Theory Dynam. Systems}, {\bf 12} (1992), no. 2, 233--253.

\bibitem{HM03} M. Hirayama, Periodic probability measures are dense in the set of invariant measures, {\it Discrete Contin. Dyn. Syst.}, {\bf 9} (2003), no. 5, 1185--1192.  

\bibitem{Ka} A. Katok, Lyapunov exponents, entropy and periodic orbits for diffeomorphisms, {\it Inst. Hautes \'Etudes Sci. Publ. Math.}, {\bf 51} (1980), 137--173.

\bibitem{LGW05} M. Li, S. B. Gan, and L. Wen, Robustly transitive singular sets via approach of an extended linear {P}oincar\'{e} flow, {\it Discrete Contin. Dyn. Syst.}, {\bf 13} (2005), no. 2, 239--269.

\bibitem{LLL24} M. Li, C. Liang, and X. Z. Liu, A closing lemma for non-uniformly hyperbolic singular flows, {\it Comm. Math. Phys.}, {\bf 405} (2024), no. 8, Paper No. 195, 35.

\bibitem{LSWW20} M. Li, Yi Shi, S. R. Wang, and X. D. Wang, Measures of intermediate entropies for star vector fields, {\it Israel J. Math.}, {\bf 240} (2020), no. 2, 791--819.

\bibitem{LY75} T. Y. Li and J. A. Yorke, Period three implies chaos, {\it Amer. Math. Monthly}, {\bf 82} (1975), no. 10, 985--992.

\bibitem{Lian} Z. Lian and L. S. Young, Lyapunov exponents, periodic orbits, and horseshoes for semiflows on Hilbert spaces, {\it J. Amer. Math. Soc.}, {\bf 25} (2012), no. 3, 637--665.

\bibitem{LLS14} C. Liang, G. Liao, and W. X. Sun, A note on approximation properties of the Oseledets splitting, {\it Proc. Amer. Math. Soc.}, {\bf 142} (2014), no. 11, 3825--3838.

\bibitem{LLS09} C. Liang, G. Liu, and W. X. Sun, Approximation properties on invariant measure and Oseledec splitting in non-uniformly hyperbolic systems, {\it Trans. Amer. Math. Soc.}, {\bf 361} (2009), no. 3, 1543--1579.

\bibitem{LST79A} S. T. Liao, A basic property of a certain class of differential systems, {\it Acta Math. Sinica}, {\bf 22} (1979), no. 3, 316--343.

\bibitem{LST79} S. T. Liao, An existence theorem for periodic orbits, {\it Beijing Daxue Xuebao}, (1979), no. 1, 1--20.

\bibitem{LST80} S. T. Liao, On the stability conjecture, {\it Chinese Annals of Mathematics Series A}, {\bf 1} (1980), no. 1, 8--30.

\bibitem{LST81} S. T. Liao, Obstruction sets. {II}, {\it Beijing Daxue Xuebao}, (1981), no. 2, 1--36.

\bibitem{LST85} S. T. Liao, Some uniformity properties of ordinary differential systems and a generalization of an existence theorem for periodic orbits, {\it Beijing Daxue Xuebao}, (1985), no. 2, 1--19.

\bibitem{Liao89} S. T. Liao, On $(\eta,d)$-contractible orbits of vector fields, {\it Systems Sci. Math. Sci.}, {\bf 2} (1989), no. 3, 193--227.

\bibitem{Lor63} E. N. Lorenz, Deterministic nonperiodic flow, {\it J. Atmospheric Sci.}, {\bf 20} (1963), no. 2, 130--141.

\bibitem{LW2025} Y. S. Lu and W. L. Wu, The Lyapunov exponents of hyperbolic measures for $C^1$ Star vector fields on three-dimensional manifolds, {\it preprint}, (2025), ArXiv:2507.23605.

\bibitem{M22}  X. Ma, Existence of periodic orbits and horseshoes for semiflows on a separable Banach space, {\it Calc. Var. Partial Differential Equations}, {\bf 61} (2022), no. 6, Paper No. 217, 37.

\bibitem{MR} R. Ma\~{n}\'{e}, An ergodic closing lemma, {\it Ann. of Math. (2)}, {\bf 116} (1982), no. 3, 503--540.

\bibitem{MM08} R. Metzger and C. Morales, Sectional-hyperbolic systems, {\it Ergodic Theory Dynam. Systems}, {\bf 28} (2008), no. 5, 1587--1597.

\bibitem{MPP04}  C. A. Morales, M. J. Pacifico, and E. R. Pujals, Robust transitive singular sets for $3$-flows are partially hyperbolic attractors or repellers, {\it Ann. of Math. (2)}, {\bf 160} (2004), no. 2, 375--432.

\bibitem{OV} V. Oseledec, A multiplicative ergodic theorem. Characteristic Ljapunov, exponents of dynamical systems, {\it Trudy Moskov. Mat. Ob\v s\v C.}, {\bf 19} (1968), 179--210.

\bibitem{PaS70} J. Palis and S. Smale, {\it Structural stability theorems}, Global Analysis (Proc. Sympos. Pure Math., Vols. XIV, XV, XVI, Berkeley, Calif., 1968), Proc. Sympos. Pure Math., XIV-XVI, Amer. Math. Soc., Providence, RI, 1970, pp. 223--231.

\bibitem{P99} S. Y. Pilyugin, Shadowing in dynamical systems, Lecture Notes in Mathematics, vol. 1706, Springer-Verlag, Berlin, 1999.

\bibitem{PS70} C. Pugh and M. Shub, The $\Omega$-stability theorem for flows, {\it Invent. Math.}, {\bf 11} (1970), 150--158.

\bibitem{PS00} E. R. Pujals and M. Sambarino, Homoclinic tangencies and hyperbolicity for surface diffeomorphisms, {\it Ann. of Math. (2)}, {\bf 151} (2000), no. 3, 961--1023.

\bibitem{SGW14} Y. Shi, S. B. Gan, and L. Wen, On the singular-hyperbolicity of star flows, {\it J. Mod. Dyn.}, {\bf 8} (2014), no. 2, 191--219.

\bibitem{Sig70} K. Sigmund, Generic properties of invariant measures for Axiom A diffeomorphisms, {\it Invent. Math.}, {\bf 11} (1970), 99--109.

\bibitem{Si72} J. G. Sina\u{\i}, Gibbs measures in ergodic theory, {\it Uspehi Mat. Nauk}, {\bf 27} (1972), no. 4(166), 21--64.

\bibitem{SS70} S. Smale, {\it The $\Omega$-stability theorem}, Global Analysis (Proc. Sympos. Pure Math., Vols. XIV, XV, XVI, Berkeley, Calif., 1968), Proc. Sympos. Pure Math., XIV-XVI, Amer. Math. Soc., Providence, RI, 1970, pp. 289–297.

\bibitem{WW10} Z. Q. Wang and W. X. Sun, Lyapunov exponents of hyperbolic measures and hyperbolic periodic orbits, {\it Trans. Amer. Math. Soc.}, {\bf 362} (2010), no. 8, 4267--4282.

\bibitem{WW19} X. Wen and L. Wen, A rescaled expansiveness for flows, {\it Trans. Amer. Math. Soc.}, {\bf 371} (2019), no. 5, 3179--3207.

\bibitem{WYZ} Wu, W. L. and Yang, D. W. and Zhang, Y., On the growth rate of periodic orbits for vector fields, {\it Adv. Math.}, {\bf 346} (2019), 170--193.
	
\end{thebibliography}

\noindent\text{Sun Qimai}\\
School of Mathematics and Statistics\\
Jiangsu Normal University, Xuzhou, 221116, P.R. China\\
sunqimajsnu@163.com\\

\noindent\text{Wang Guangwa}\\
School of Mathematics and Statistics\\
Jiangsu Normal University, Xuzhou, 221116, P.R. China\\
wanggw7653@163.com\\

\noindent\text{Wanlou Wu}\\
School of Mathematics and Statistics\\
Jiangsu Normal University, Xuzhou, 221116, P.R. China\\
wuwanlou@163.com\\

\end{document}